\documentclass[11pt,a4paper,reqno]{amsart}

\usepackage[mathscr]{euscript}
\usepackage{amssymb}
\usepackage{amsthm}
\usepackage{amsxtra}
\usepackage{bbm}
\usepackage{mathtools}
\usepackage{stmaryrd}

\usepackage{extarrows}
\usepackage{tikz-cd}
\usepackage{url}

\numberwithin{equation}{section}

\usepackage[margin=1.05in]{geometry}

\newtheorem{theorem}{Theorem}[section]
\newtheorem{proposition}[theorem]{Proposition}
\newtheorem{lemma}[theorem]{Lemma}

\newtheorem{corollary}[theorem]{Corollary}

\theoremstyle{definition}
\newtheorem{definition}[theorem]{Definition}

\theoremstyle{remark}

\newtheorem*{acknowledgements}{Acknowledgements}

\title[Heat equation and stable minimal Morse functions]{Heat equation and stable minimal Morse functions on real and complex projective spaces}

\author{Sebasti\'{a}n Mu\~{n}oz Mu\~{n}oz and Alexander Quintero V\'{e}lez} 
  
\address{Departamento de Matem\'{a}ticas \\
Universidad del Valle\\
Calle 13 \# 100-00\\
Cali, Colombia}
 \email{munoz.sebastian@correounivalle.edu.co}
 
\address{Escuela de Matem\'{a}ticas\\
Universidad Nacional de Colombia Sede Medell\'{i}n\\
Calle 59a \# 63-20\\
Medell\'{i}n, Colombia}
 \email{aquinte2@unal.edu.co}

\begin{document}

\begin{abstract}
Following similar results in \cite{rmk} for flat tori and round spheres, in this paper is presented a proof of the fact that, for ``arbitrary'' initial conditions $f_0$, the solution $f_t$ at time $t$ of the heat equation on real or complex projective spaces eventually becomes (and remains) a minimal Morse function. Furthermore, it is shown that the solution becomes stable. 
\end{abstract}

 \maketitle

\section{Introduction}\label{intro}

The heat equation, which is the mathematical model for heat flow, and in its simplest form is written as 
$$\frac{\partial f}{\partial t} = \Delta f ,$$
where $\Delta$ is the Laplace-Beltrami operator, is among the most relevant partial differential equations in mathematics and physics, and it has truly remarkable properties. For instance, if an ``arbitrary" initial condition $f_0$ is subjected to this equation, then the solution $f_t$, for any instant of time $t>0$, will be a $C^{\infty}$ function. In the light of such facts, the heat equation presents itself as a fundamental smoothing process.

In a recent paper \cite{rmk}, a result in this direction was established, further evidencing important smoothing properties of this equation. Given a smooth function $f$ defined on a manifold $M$, an indicator of regularity for $f$ is whether all of its critical points are non-degenerate, which is to say that $f$ is a Morse function. The authors showed (see Theorem $2.1$ and Theorem $3.1$ of that paper) that, at least in the case of $n$-dimensional round spheres or flat tori, ``arbitrary" smooth initial conditions $f_0$ are eventually transformed by the heat equation into minimal Morse functions; Morse functions that have the smallest possible number of critical points in the manifold. 

Investigating the question of whether this is a more general phenomenon presenting itself in other compact Riemannian manifolds, this paper presents  proof of the analogous result for real and complex projective spaces of arbitrary dimension. We will further observe that in these spaces, the heat process endows the solution $f_t$ with a fundamental property called \emph{stability}, which roughly speaking means that functions close to $f_t$ will be  ``identical" to $f_t$ modulo a suitable change of coordinates\footnote{Close in the standard sense of the Whitney $C^{\infty}$ topology, which measures functions by their size and that of their partial derivatives of all orders.}. This is made precise by the following definition.

\begin{definition}
Let $M$ be a compact smooth manifold, and let $f \in C^\infty(M)$. $f$ is said to be \emph{stable} if there exists a neighborhood $W_f$ of $f$ in the Whitney $C^\infty$ topology such that for each $f'\in W_f$ there exist diffeomorphisms $g,h$ such that the following diagram commutes:
$$
  \begin{tikzcd}
    M \arrow{r}{f} \arrow[swap]{d}{g} & \mathbb{R} \arrow{d}{h} \\
    M \arrow{r}{f'} & \mathbb{R}
  \end{tikzcd}
$$
\end{definition}

The corollary to the following fundamental theorem (see \cite[pp. 79-80]{stability}) gives a simple characterization of stable functions which will be key to what follows.   
\begin{theorem}[Stability Theorem] \label{stab}
Let $M$ be a compact smooth manifold, and $f \in C^{\infty}(M)$. Then $f$ is a Morse function with distinct critical values if and only if it is stable.
\end{theorem}

\begin{corollary} \label{cor}
If $M$ is a smooth manifold, and $f$ is a Morse function with distinct critical values, then there exists a neighborhood of $f$ in the $C^\infty$ topology such that $g$ is a Morse function with distinct critical values and the same number of critical points as $f$ for all $g$ in said neighborhood. In particular, if $M$ is compact (so that the $C^{\infty}$ has the union of all $C^r$ topologies as a basis, each a Banach space \cite{hirsch}), there exist $r$ and $\varepsilon>0$ such that $g$ is a Morse function with distinct critical values and the same number of critical points as $f$ whenever $\|f-g\|_r< \varepsilon$ , with $\| \cdot \|_r$ being a fixed norm for the $C^r$ topology.
\end{corollary}

Now, our precise formulation of an ``arbitrary" smooth initial condition $f_0$ will be that it belongs to a fixed open and dense set $S$ in the  $C^{\infty}$ topology. We can now state the two main results.

\begin{theorem}\label{real} There exists a set $S \subset C^{\infty}(\mathbb{RP}^n)$, that is dense and open in the $C^{\infty}$ topology, such that for any initial condition $f=f_0\in S$, if $f_t$ is the corresponding solution to the heat equation on  $\mathbb{RP}^n$ at time $t$, then there exists $T>0$ such that for $t\geq T$, $f_t$ is a stable minimal Morse function on $\mathbb{RP}^n$ .
\end{theorem}

\begin{theorem} \label{complex}
There exists a set $S \subset C^{\infty}(\mathbb{CP}^n)$, that is dense and open in the $C^{\infty}$ topology, such that for any initial condition $f_0\in S$, if $f_t$ is the corresponding solution to the heat equation on  $\mathbb{CP}^n$  at time $t$, then there exists $T>0$ such that for $t\geq T$, $f_t$ is a stable minimal Morse function on  $\mathbb{CP}^n$. 
\end{theorem}

The basic strategy for the proof on both spaces may be sketched as follows: On compact Riemannian manifolds, as is well known, the solution has the form
$$f_t=h_0+h_1e^{-\lambda_1 t} +h_2 e^{-\lambda_2 t}+ \cdots,$$
where $\lambda_1<\lambda_2<\cdots$ are the non-trivial eigenvalues for the Laplace-Beltrami operator, and the $h_i$ are the projections of $f_0$ onto the corresponding eigenspaces. The overall idea is to exploit the exponentially decaying terms of the solution to approximate $f_t$ by the first two terms of the sum for large times. Our set $S$ will then consist of those functions for which these first two terms add up to a stable, minimal Morse function. Since stability is by definition a property unchanged by small perturbations, it is apparent that the desired result follows provided that 
\begin{itemize}
\item[(i)] $S$ is dense and open, and
\item[(ii)] we successfully bound the remaining terms to make the approximation work. 
\end{itemize}

In section \ref{basics}, simple conditions for the main result to hold in a manifold will be established, as well as a technical statement that will be needed later to verify the above-stated conditions on the two spaces in question. On the other hand, sections \ref{realsection} and \ref{complexsection} are, respectively, each dedicated to the proofs of Theorems \ref{real} and \ref{complex}.

\section{Basic Results} \label{basics}
The following lemma gives two concrete sufficient conditions for items (i) and (ii) of the introduction to be true, and thereafter makes rigorous the previously sketched strategy to prove (using some simple estimates and the profound characterization of stability given by Theorem \ref{stab}) that the desired theorem will follow on compact Riemannian manifolds in which the eigenvalues grow at least linearly as soon as we have these two sufficient conditions. 

\begin{lemma}\label{basic} Let $M$ be a compact, connected Riemannian manifold, let $0=\lambda_0<\lambda_1 < \lambda_2 < \cdots$ be the distinct eigenvalues for the Laplace-Beltrami operator and let $\mathcal{B}=\{\varphi_i \mid 1 \leq i \leq d_1\}$ be any basis (not necessarily orthonormal) for the $\lambda_1$-eigenspace. Suppose the eigenvalues grow at least linearly\footnote{By this we mean that there exists $r>0$ such that $\lambda_{j}>rj.$}, that the $0$-eigenspace is trivial\footnote{This means that it consists only of constant functions. This is known to hold on real and complex projective spaces, basically because it holds on the sphere. This condition is not essential at all; see the next footnote.}, and the following two conditions hold:
\begin{itemize}
\item[\bf(C1)] The set $B$ of $d_1$-tuples $(c_1,\ldots,c_{d_1})$ such that $\sum_i c_i \varphi_i$ is a minimal Morse function on $M$ with distinct critical values is an open dense subset of $\mathbb{R}^{d_1}$.
\item[\bf(C2)] For each $f\in C ^{\infty}(M)$, there exist $N$, $C$ such that the projection $h_j=\pi_j(f)$ of $f$ onto the $j$th eigenspace satisfies
$$\|h_j\|_r \leq C(1+j^N).$$
\end{itemize}
Then there exists a set $S \subset C^{\infty}(M)$, that is dense and open in the $C^{\infty}$ topology, such that for any initial condition $f_0\in S$, if $f_t$ is the corresponding solution to the heat equation on  $M$ at time $t$, then there exists $T>0$ such that for $t\geq T$, $f_t$ is a minimal Morse function with distinct critical values on $M$.
\end{lemma}

\begin{proof}
Let $S$ be the set of functions $f$ whose projection $\pi_1(f)$ onto the $\lambda_1$-eigenspace is a minimal Morse function with distinct critical values. Let $f \in S$. By compactness, $\| \cdot \|_{L^2(M)} = O(\| \cdot \|_{0}) $, so functions in a sufficiently small neighborhood $U$ of $f$ in the $C^0$ topology (which is contained in the $C^{\infty}$ topology) will have its Fourier coefficients (with respect to any fixed orthonormal basis) as close as desired to those of $f$. Observing that the coefficients with respect to $\mathcal{B}$ are continuous functions of the Fourier coefficients for the $\lambda_1$-eigenspace, condition {\bf(C1)} implies that if $U$ is small enough, $U \subset S$, hence $S$ is open.  Now let $f\in C^{\infty}(M)$, and let $g$ be obtained from $f$ such that $\pi_i(g)=\pi_i(f)$ ($i \neq 1$) and $\pi_1(g)$ comes from slightly modifying the coefficients of $\pi_1(f)$ with respect to $\mathcal{B}$ so that $\pi_1(g)$ is a minimal Morse function with distinct critical values (this is again possible by condition {\bf(C1)}). By compactness of $M$, if the modification is slight enough, $g$ will be as close as desired to $f$ in any $C^k$ (and so in the $C^{\infty}$) topology. Thus, $S$ is dense. Next we check that if $f\in S$ and $f=h_0+h_1+\cdots$ with $h_i=\pi_i(f)$, then $f_t=h_0+e^{-\lambda_1t} h_1 + \cdots$ is a Morse minimal function with distinct critical values for large enough $t$. Since $h_0$ is constant it suffices to show the same for $(f_t-h_0)e^{\lambda_1 t}=h_1 +  e^{(\lambda_1- \lambda_2)t} h_2+ \cdots$, and by Corollary \ref{cor} it is enough to prove that for each $k$, $\|(f_t-h_0)e^{\lambda_1 t}-h_1\|_k \rightarrow 0$ as $t \rightarrow \infty$. One has
\begin{align*} 
\|(f_t-h_0)e^{\lambda_1 t}-h_1\|_k =& \|h_2 e^{(\lambda_1-\lambda_2)t} + \cdots\|_k \\
=& e^{(\lambda_1-\lambda_2)t} \left\|\sum_{j\geq 2} e^{(\lambda_2-\lambda_j)t}h_j \right\|_k \\
\leq& e^{(\lambda_1-\lambda_2)t} \sum_{j\geq 2} e^{(\lambda_2-\lambda_j)t}\|h_j\|_k \\
\leq &  e^{(\lambda_1-\lambda_2)t}C\sum_{j\geq 2} (1+j^N)  e^{(\lambda_2-\lambda_j)t}. 
 \end{align*}
Because the $\lambda_j$ grow at least linearly, the series on the right is clearly convergent and a decreasing function of $t$, and the first factor tends to zero as $t \rightarrow \infty$, so this completes the proof.\footnote{It is apparent that, at the cost of making the notation more unwieldly, one can prove completely analogous conditions in which one approximates $f_t$ by any finite-dimensional space of eigenfuctions instead of just the first non-trivial one.} 
\end{proof} 

Since the real and complex projective spaces, besides conditions {\bf(C1)} and {\bf(C2)}, are easily seen to satisfy the hypotheses of the above lemma,\footnote{The growth property for the eigenvalues is obvious because they are integers; this fact will be remarked later.} the remaining majority of this paper is dedicated to proving that the aforementioned conditions hold on these spaces, from which the theorems will follow. We end this section with a technical statement that will be needed later for this purpose.

\begin{lemma} \label{lie}
Let $M$ be a smooth manifold and let $G$ be a Lie group acting smoothly, freely and properly on $M$. Let $h \in C^{\infty}(M)$ be constant on each $G$-orbit, so that it descends to $\overline{h} \in C^{\infty}(M/G)$. Then there exist smooth coordinates $(x_1,\ldots,x_n, y_1,\ldots, y_k)$ for $M$ such that $(y_1,\ldots,y_k)$ are coordinates for $M/G$ and 
$$
\frac{\partial^{i_1+\cdots+i_s} h}{\partial y_{i_1}\cdots \partial y_{i_s}}=\frac{\partial^{i_1+\cdots+i_s} \overline{h}}{\partial y_{i_1}\cdots \partial y_{i_s}}
$$
for any indices $i_1,\ldots, i_s$.
\end{lemma}

\begin{proof} In Theorem 21.10 of \cite{smooth} one sees that there exist cubic coordinates $(x_1,\ldots,x_n, y_1,\ldots, y_k)$ for $M$ that intersect each $G$-orbit in either the empty set or a slice $(y_1,\ldots,y_k)=(c_1,\ldots,c_k)$ and such that $(y_1,\ldots,y_k)$ are coordinates for $M/G$. In these coordinates it is clear that
$$h(x_1,\ldots,x_n, y_1,\ldots, y_k)=\overline{h}(y_1,\ldots, y_k),$$
and the statement follows by differentiating both sides.
\end{proof}

\section{Real projective spaces} \label{realsection}
In this section we demonstrate our result for real projective spaces  $\mathbb{RP}^n$.  By construction, the space of functions on  $\mathbb{RP}^n$ (realized as $S^n/G$, with $G$ being the two element group generated by the antipodal map) is identified (via pullback) with the functions in $S^n$ for which 
 $$f(x)=f(-x).$$
Therefore, one sees that the eigenfunctions for the Laplace-Beltrami operator in this manifold are precisely the ones on $S^n$ that satisfy $f(x)=f(-x)$. That is, the eigenvalues are $\lambda_{j}=2j(2j+n-1)$, with the corresponding eigenspace being the space of homogeneous harmonic polynomials of degree $2 j$ in $n+1$ variables \cite{spectre}. The following proposition is then a detailed analysis of the first non-trivial eigenspace, whose ultimate purpose is the verification of condition {\bf(C1)} of Lemma \ref{basic}.

\begin{proposition}\label{prop1}
Let $f(x)=\sum_{i,j} a_{ij}x_i x_j$ be a real quadratic form in $n+1$ variables, with $A=(a_{ij})$ being a symmetric matrix.\footnote{It is easily seen that any quadratic form may be expressed in this symmetric fashion.}  Then $f$ is a minimal Morse function on  $\mathbb{RP}^n$ with distinct critical points if and only if the matrix $A=(a_{ij})$ has distinct eigenvalues.
\end{proposition}

\begin{proof}
We may diagonalize this quadratic form through an orthogonal change of coordinates $y=Bx$, and write $f(y)=\lambda_1 y_1^2 + \cdots +\lambda_{n+1} y_{n+1} ^2$, where $\lambda_i$ are the  eigenvalues of $A$. Write $D=\text{diag}(\lambda_1,\ldots,\lambda_{n+1})$. The gradient of $f$ in these coordinates (seen as a function in $\mathbb{R}^{n+1}$) is then $\nabla{f}(y)=2Dy$.  The proof will be done once we show that $f$ has exactly $n+1$  critical points in $\mathbb{RP}^n$, all of them being non-degenerate,   if and only if $A$ has distinct eigenvalues.\footnote{Since the Morse-Smale Characteristic of $\mathbb{RP}^n$, the minimum number of critical points a Morse function can have in this space, is $n+1$ \cite{andrica}.}
Since  $\mathbb{RP}^n$ is the quotient of $S^n$ by a properly discontinuous action, it suffices to count the number $j$ of critical points for $f$ on the sphere, and then counting the resulting equivalence classes. Since every point on the sphere has orbit two under the antipodal map, the answer will be $j/2$, thus we need to show that $j=2(n+1)$. Now, the critical points of $f$ on the sphere will be those $y$ that satisfy
$$\nabla{f}(y)\cdot v=0 \text{ for all } v\in T_y S^n=\langle y \rangle ^{\perp},$$
which means 
$$2Dy \cdot v=0 \text{ for all } v\in T_y S^n=\langle y \rangle ^{\perp},$$
 or equivalently 
 $\langle y \rangle ^{\perp} \subset \langle Dy \rangle ^{\perp},$
and this just says that $Dy$ is a scalar multiple of $y$. Since $D$ is just the matrix of $A$ in the $y$ coordinates, the desired critical points are precisely the eigenvectors of the transformation $A$ that lie on the sphere.
  
If $A$ had a repeated eigenvalue, then by intersecting a two-dimensional eigenspace with the sphere we see that there would be an entire circumference of critical points, implying that $f$ is not a Morse function on the sphere (since a Morse function has only isolated critical points), so this proves the ``only if" part of the statement.
  
Now suppose the eigenvalues are all distinct, so all eigenspaces are straight lines. Then the desired critical points are precisely the $2(n+1)$ intersections of these lines with the sphere (two per line), which in our choice of coordinates are just the canonical basis vectors $\{\pm e_1,\ldots,\pm e_{n+1}\}$, so we get $j=2(n+1)$, as wanted. The critical values are distinct because they are precisely the $\lambda_i$. Now, it remains to see that these critical points $\pm e_i$ are all non-degenerate. To prove this, we compute the Hessian in local coordinates and verify that it is invertible. Fix $i$, and write $f$ in the local coordinate chart for $S^n$ around $\pm e_i$ defined by $\varphi(y_1,\ldots,y_{n+1})=(y_1,\ldots,y_{i-1},y_{i+1},\ldots,y_{n+1})=(u_1,\ldots,u_n)$. One gets:
  $$f(u_1,\ldots,u_n)=\lambda_i + \sum_{j\neq i} (\lambda_j-\lambda_i)u_j^2$$
Therefore the gradient is given by
  $$\nabla f(u)=2\text{diag}(\lambda_1-\lambda_i,\ldots,\lambda_{i-1}-\lambda_i,\lambda_{i+1}-\lambda_i,\ldots ,\lambda_{n+1}-\lambda_i)u$$
and so the Hessian of $f$ at $e_i$ in these coordinates is the matrix
  $$2\text{diag}(\lambda_1-\lambda_i,\ldots,\lambda_{i-1}-\lambda_i,\lambda_{i+1}-\lambda_i,\ldots ,\lambda_{n+1}-\lambda_i),$$
which is invertible because the $\lambda_i$ are distinct. \end{proof}   
 
This section now ends with the proof of the first of the two main results.

\begin{proof}[Proof of Theorem \ref{real}] To prove this theorem we demonstrate that conditions {\bf(C1)} and {\bf(C2)} of Lemma \ref{basic} are satisfied. By the remarks at the beginning of the section, the generic function in the $\lambda_1$-eigenspace has the form of $f$ in Proposition \ref{prop1}, with the additional condition that $\mathrm{tr}(A)=0$, or $a_{n+1,n+1}=-(a_{11}+\cdots + a_{nn})$. Therefore the following set is a basis for this eigenspace:
 $$
 \mathcal{B}=\left\{ \tfrac{1}{2}x_i x_j\, \big\vert \, 1\leq i<j \leq n+1 \right\} \cup \left\{ x_{ii}^2-x_{n+1,n+1}^2 \mid 1\leq i<n+1 \right\}.
 $$
The coefficients of $f$ for this basis (of size $(n^2+3n)/2$) are the $a_{ij}$ for $i<j$ or $i=j<n+1$. Diagonalizing $A=Q^T \text{diag}(\lambda_1,\ldots,\lambda_{n+1}) Q$ one sees that by leaving $B$ fixed and slightly varying the $\lambda_i$ (while keeping the tracelessness condition), one can make the eigenvalues distinct, so by continuity (given by the matrix equation) and Proposition \ref{prop1}, the set $B$ of condition {\bf(C1)} in Lemma \ref{basic} corresponding to our $\mathcal{B}$ is dense in $\mathbb{R}^{(n^2+3n)/2}$. 
 
To show that $B$ is open, we note first that the coefficients of the characteristic polynomial $p$ of $A$ are continuous functions of the $a_{ij}$ (in the usual sense), and the $\lambda_i$ depend continuously on the coefficients of $p$ in the following sense: if we consider disjoint $\varepsilon$-neighborhoods of the $\lambda_i$, then the zeroes of $p_t$ will, in some order (say $\beta_1,\ldots,\beta_{n+1})$, satisfy $| \lambda_i - \beta_i |<\varepsilon$, whenever $p_t$ is obtained through small enough variations of the coefficients of $p$. This is readily seen to be a corollary to Rouche's Theorem,\footnote{Rouche's Theorem is a standard complex-analytic result that, roughly speaking, states that the number of zeroes of two analytic functions inside a bounded region is invariant if these functions are close enough in the boundary of the region.} and it (together with the fact that $A$ is symmetric which forces all roots to be real) implies that if $p$ has distinct real roots, small variations of the $a_{ij}$ will not affect this property, as desired.
 
Now to check condition {\bf(C2)}, let $\overline{h}_{2j}=\pi_j(f)$ be the projection of $f$ onto the $j$th eigenspace, with $h_{2j}$ being the pulled back homogeneous harmonic polynomial of degree $2j$. By Lemma \ref{lie}, $\|\overline{h}_j\|_k = O(1)\|h_j\|_k.$ From Section 10 in \cite{garret}, we then obtain
$$
\|\overline{h}_{2j}\|_k =O(1) \binom{n+2j}{n}^{1/2}\left(1+\lambda_j^{1/2}\right)^k\|h_{2j}\|_{L^2(S^n)}.
$$
 Since both $\lambda_j=2j(2j+n-1) $ and $\binom{n+j}{n}$ are polynomials in $j$ (of degrees $2$ and $n$ respectively), and  $\|h_{2j}\|_{L^2(S^n)} \leq \|f \|_{L^2(S^n)}$, the desired inequality then follows immediately, completing the proof.
 \end{proof}
 
\section{Complex projective spaces} \label{complexsection}
This section is dedicated to proving the desired Theorem \ref{complex} for the complex projective space $\mathbb{CP}^n$. We henceforth identify $\mathbb{C}^n$ with $\mathbb{R}^{2n}$ as real vector spaces and manifolds by the correspondence
$$
(z_1,\ldots, z_n)=(x_1+\sqrt{-1} y_1, \ldots ,x_n+\sqrt{-1} y_n ) \equiv (x_1,y_1,\ldots, x_n,y_n).
$$
The Riemannian manifold $\mathbb{CP}^n$ with the Fubini-Study metric may then be realized as the quotient $S^{2n+1}/\mathrm{U}(1)$, and  the (complex-valued) eigenfunctions for the Laplacian are the eigenfunctions on the sphere that are invariant under multiplication by elements of $\mathrm{U}(1)$. The most convenient way to characterize them for the matter at hand is as the bi-homogeneous harmonic polynomials in $(z,\overline{z})$ of bi-degree $(k,k)$ ($k=0,1,2,\ldots$), with corresponding eigenvalues $4k(n+k)$ (cf. \cite{spectre}). Now, this implies that the real-valued eigenfunctions for the first eigenvalue $4(n+1)$ are precisely the functions $f$ described in the proposition below when $A$ is traceless, since the condition of being real-valued forces the matrix of coefficients to be Hermitian. This result then plays a role analogous to that of Proposition \ref{prop1} in the previous section.

\begin{proposition}\label{prop2}
Let $f\colon \mathbb{C}^{n+1} \backslash \{0\} \rightarrow \mathbb{R}$ be defined by $$f(z)=\frac{\sum_{i,j} a_{ij}z_i \overline{z_j}}{|z|^2},$$  with $A=(a_{ij})$ being a Hermitian matrix.  Then $f$ descends to a minimal Morse function with distinct critical values on  $\mathbb{CP}^n$ (realized as the quotient of $\mathbb{C}^n \backslash \{0\}$ obtained by identifying points on the same line) if and only if the matrix $A=(a_{ij})$ has distinct eigenvalues.
\end{proposition}

\begin{proof} It is clear that $f$ is 0-homogeneous, so we may consider it as a function on $\mathbb{CP}^n$. Similarly to the proof of Proposition \ref{prop1}, the fact that $A$ is Hermitian allows us to use the Spectral Theorem to diagonalize $A$ through a unitary change of coordinates $w=Bz$. In these coordinates $f$ takes the form 
$$
f(w)=\frac{\sum_{i=1}^{n+1} \lambda_i |w_i|^2}{|w|^2},
$$
where $\lambda_i$ are the eigenvalues of $A$. Replacing $B$ with a permutation of its columns allows us to assume $\lambda_1 \leq \ldots \leq \lambda_{n+1}$. The sum of the Betti numbers for $\mathbb{CP}^n$ is $n+1$, so it suffices to show that $f$ has $n+1$ critical points on $\mathbb{CP}^n$, and that they are all non-degenerate. The sets $U_i=\{[w]\in \mathbb{CP}^n \mid w_i \neq 0\}$ together with the inhomogeneous coordinate functions $w=(w_1,\ldots,w_{i-1},w_{i+1},\ldots,w_n)$ are a family of smooth charts covering $\mathbb{CP}^n$, therefore we may proceed by computing the critical points in each $U_i$ through the use of the corresponding local coordinates. In these inhomogeneous coordinates, $f$ then takes the form:
$$f(w_1,\ldots,w_{i-1},w_{i+1},\ldots,w_n)=\frac{\lambda_i + \sum_{j \neq i} \lambda_j |w_j|^2}{1+\sum_{j\neq i}|w_j|^2}=\frac{\lambda_i + \sum_{j \neq i} \lambda_j x_j^2+\lambda_j y_j^2}{1+\sum_{j\neq i}x_j^2+y_j^2}.
$$
This gives
\begin{align*}
\mathrm{d}f =& 
-\frac{\sum_{j\neq i} 2x_j \mathrm{d}x_j + 2 y_j \mathrm{d}y_j}{\left(1+\sum_{j\neq i}x_j^2+y_j^2\right)^2}\left(\lambda_i + \sum_{j \neq i} \lambda_j x_j^2+\lambda_j y_j^2 \right)+\frac{\sum_{j \neq i} 2\lambda_j x_j \mathrm{d}x_j+\lambda_j 2y_j \mathrm{d}y_j}{1+\sum_{j\neq i}x_j^2+y_j^2}  \\ 
=& \frac{\sum_{j\neq r} 2x_j\left(\lambda_j - \lambda_i + \sum_{r\neq i} (\lambda_j - \lambda_r)(x_r^2+y_r^2)\right)\mathrm{d}x_j}{\left(1+\sum_{j\neq i}x_j^2+y_j^2\right)^2} \\
&\qquad \qquad\qquad \qquad\qquad + \frac{\sum_{j\neq r}2y_j\left(\lambda_j -\lambda_i + \sum_{r\neq i} (\lambda_j - \lambda_r)(x_r^2+y_r^2)\right)\mathrm{d}y_j}{\left(1+\sum_{j\neq i}x_j^2+y_j^2\right)^2} .
\end{align*}
So $\mathrm{d}f=0$ means that for each $j\neq i$,
$$
x_j\left(\lambda_j - \lambda_i + \sum_{r\neq i} (\lambda_j - \lambda_r)(x_r^2+y_r^2)\right)=0,
$$
and
$$
y_j\left(\lambda_j -\lambda_i + \sum_{r\neq i} (\lambda_j - \lambda_r)(x_r^2+y_r^2)\right)=0.
$$
If $\lambda_i$ was a repeated eigenvalue, say $\lambda_i=\lambda_{i+1}$, then letting $x_1=\ldots=x_{i-1}=x_{i+2}=\ldots=x_n=0$ and $y_1=\ldots=y_{i-1}=y_{i+2}=\ldots=y_{n+1}=0$ it follows that the equations above hold for any choice of $x_{i+1}$, $y_{i+1}$, so the critical points of $f$ are not isolated and therefore it is not a Morse function; the ``only if" part of the proposition is thus proved.

Conversely, suppose $\lambda_1 < \ldots < \lambda_{n+1}$. Setting $j=1$ in the previous two equations and using the fact that the $\lambda_r$ are in increasing order, it follows that $x_1=y_1=0$. Inductively one gets $x_2,y_2,\ldots,x_{i-1},y_{i-1}=0$. Similarly starting with $j=n+1$ and inducting backwards one concludes $x_j=y_j=0$ for each $j\neq i$. Hence the unique critical point of $f$ in $V_i$ is the point $[0,\ldots,1,\ldots,0]$, where the $1$ is in the $i$th position. So there are $n+1$ critical points in total, and their corresponding values are distinct (they are the $\lambda_i$). To check non-degeneracy, it follows easily from the previous computation that the determinant of the Hessian at $[0,\ldots,1,\ldots,0]$ in the $i$th inhomogeneous coordinates is just 
$$2^{2n}(\lambda_1-\lambda_i)^2\cdots (\lambda_{i-1}-\lambda_i)^2(\lambda_{i+1}-\lambda_i)^2 \cdots (\lambda_{n+1}-\lambda_i)^2,$$
which is non-zero because the $\lambda_i$ are distinct.
 \end{proof}

The previous analysis now allows us to prove the second and final theorem.

\begin{proof}[Proof of Theorem \ref{complex}] As in the proof of Theorem \ref{real}, we must verify conditions {\bf(C1)} and {\bf(C2)} of Lemma \ref{basic}. This time the generic function in the $\lambda_1$-eigenspace has the form of $f$ in Proposition \ref{prop2}, with the condition $\mathrm{tr}(A)=0$, or $a_{n+1,n+1}=-(a_{11}+\cdots + a_{nn})$. Write $a_{ij}=b_{ij}+\sqrt{-1}c_{ij}$. Since $a_{ji}=\overline{a_{ij}}$, writing 
 $$a_{ij}z_i \overline{z_j} + a_{ji}z_j \overline{z_i}=b_{ij}(z_i \overline{z_j} +z_j \overline z_i) + \sqrt{-1}c_{ij} (z_i \overline{z_j}  - z_j \overline{z_i}),$$
one sees that the dimension in this case is $n(n+2)$ elements and a basis is:
 \begin{align*}
 &\left\{ z_i \overline{z_j} + z_j \overline{z_i} \mid 1\leq i<j \leq n+1\right\} \cup \left\{\sqrt{-1}(z_i \overline{z_j} - z_j \overline{z_i})\mid 1\leq i<j \leq n+1\right\}\\
 &\qquad\cup\left\{ z_{ii}^2-z_{n+1,n+1}^2 \mid 1\leq i<n+1 \right\}.\end{align*}
 The coefficients of $f$ for this basis are the $b_{ij}$ and the $c_{ij}$ for $i<j$ or $i=j<n+1$. Noticing that these coefficients and the $a_{ij}$ are bijectively and bicontinuously related, we may diagonalize $A=Q^* \text{diag}(\lambda_1,\ldots,\lambda_{n+1}) Q$ and finish verifying the same condition {\bf(C1)} with the same argument as that of Theorem \ref{real}. Now, for condition {\bf(C2)}, the argument in Theorem \ref{real} (again appealing to Lemma \ref{basic}) gives the estimate
 $$
 \|\overline{h}_{2j}\|_k =O(1) \binom{2n+1+2j}{2n+1}^{1/2}\left(1+\lambda_j^{1/2}\right)^{k}\|h_{2j}\|_{L^2(S^{2n+1})}
 $$
 for the projection of $f$ onto the $j$th eigenspace, and the result follows in the same way since $\lambda_j=4j(n+j)$ and $\binom{2n+1+2j}{2n+1}$ are polynomials (of degrees $2$ and $2n+1$) in $j$.
\end{proof}

\begin{acknowledgements}
We sincerely thank professor Gonzalo Garc\'ia Camacho for his unconditional support and valuable advice during the development of this work.
\end{acknowledgements}


\addcontentsline{toc}{section}{Bibliography}
\end{document}